\documentclass[12pt,a4paper]{amsart}
\usepackage{amsmath, amsthm, amsfonts, amssymb,float,graphicx,enumerate,calc,color}
\usepackage[colorlinks=true,citecolor=black,linkcolor=black]{hyperref}
\newcommand{\doi}[1]{\href{http://dx.doi.org/#1}{doi:\texttt{#1}}}

\newcommand{\urlprefix}{}
\usepackage[numbers,sort&compress]{natbib}

\theoremstyle{plain}
\newtheorem{theorem}{Theorem}[section]
\newtheorem{lemma}[theorem]{Lemma}
\newtheorem{corollary}[theorem]{Corollary}

\theoremstyle{definition}

\renewcommand{\leq}{\leqslant}
\renewcommand{\geq}{\geqslant}

\newcommand{\DEF}[1]{\textsl{#1}}       

\newcommand{\floor}[1]{\ensuremath{\protect\lfloor#1\rfloor}}

\newcommand{\G}{\mathcal{G}}

\DeclareMathOperator{\tw}{tw}
\DeclareMathOperator{\nc}{nc}
\DeclareMathOperator{\bd}{bd}

\newcommand{\lex}[2]{#1[#2]}
\newcommand{\LexVertex}[2]{#1^{(#2)}}
\newcommand{\ang}[1]{\langle{#1}\rangle}
\newcommand{\wminor}[1]{$#1$-minor} 
\newcommand{\wmodel}[1]{$#1$-model} 
\newcommand{\f}[3]{48 #1 \sqrt{#2 + #3}}  
\newcommand{\case}[2]{{\smallskip\bf Case~#1: #2}} 

\begin{document}

\title[Complete Graph Minors and the Graph Minor Structure
Theorem]{Complete Graph Minors and\\ the Graph Minor Structure
  Theorem}

\author{Gwena\"el Joret} \address{\newline D\'epartement
  d'Informatique \newline Universit\'e Libre de Bruxelles \newline
  Brussels, Belgium} \email{gjoret@ulb.ac.be}

\author{David~R.~Wood} \address{\newline Department of Mathematics and
  Statistics \newline The University of Melbourne \newline Melbourne,
  Australia} \email{woodd@unimelb.edu.au}

\thanks{\textbf{MSC Classification}: graph minors 05C83, topological
  graph theory 05C10}

\thanks{This work was supported in part by the Actions de Recherche
  Concert\'ees (ARC) fund of the Communaut\'e fran\c{c}aise de
  Belgique.  Gwena\"el Joret is a Postdoctoral Researcher of the Fonds
  National de la Recherche Scientifique (F.R.S.--FNRS).  The research
  of David Wood is supported by a QEII Research Fellowship and a
  Discovery Project from the Australian Research Council.}

\date{\today}

\begin{abstract}
  The graph minor structure theorem by Robertson and Seymour shows
  that every graph that excludes a fixed minor can be constructed by a
  combination of four ingredients: graphs embedded in a surface of
  bounded genus, a bounded number of vortices of bounded width, a
  bounded number of apex vertices, and the clique-sum operation. This
  paper studies the converse question: What is the maximum order of a
  complete graph minor in a graph constructed using these four
  ingredients? Our main result answers this question up to a constant
  factor.
\end{abstract}

\maketitle


\section{Introduction}
\label{sec:Intro}

\citet{RS-GraphMinorsXVI-JCTB03} proved a rough structural
characterization of graphs that exclude a fixed minor. It says that such a
graph can be constructed by a combination of four ingredients: graphs
embedded in a surface of bounded genus, a bounded number of vortices
of bounded width, a bounded number of apex vertices, and the
clique-sum operation. Moreover, each of these ingredients is
essential.

In this paper, we consider the converse question: What is the maximum
order of a complete graph minor in a graph constructed using these
four ingredients? Our main result answers this question up to a
constant factor.

To state this theorem, we now introduce some notation; see
Section~\ref{sec:Definitions} for precise definitions.  For a graph
$G$, let $\eta(G)$ denote the maximum integer $n$ such that the
complete graph $K_{n}$ is a minor of $G$, sometimes called the 
\DEF{Hadwiger number} of $G$.  For integers $g,p,k\geq0$, let $\G(g,p,k)$
be the set of graphs obtained by adding at most $p$ vortices, each
with width at most $k$, to a graph embedded in a surface of Euler
genus at most $g$.  For an integer $a\geq0$, let $\G(g,p,k,a)$ be the
set of graphs $G$ such that $G \setminus A\in\G(g,p,k)$ for some set
$A\subseteq V(G)$ with $|A|\leq a$. The vertices in $A$ are called
\DEF{apex} vertices.  Let $\G(g,p,k,a)^+$ be the set of graphs
obtained from clique-sums of graphs in $\G(g,p,k,a)$.

The graph minor structure theorem of \citet{RS-GraphMinorsXVI-JCTB03}
says that for every integer $t\geq1$, there exist integers
$g,p,k,a\geq0$, such that every graph $G$ with $\eta(G)\leq t$ is in
$\G(g,p,k,a)^+$. We prove the following converse result.

\begin{theorem}
  \label{thm:Main}
  For some constant $c>0$, for all integers $g,p,k,a\geq0$, for every
  graph $G$ in $\G(g,p,k,a)^+$,
$$\eta(G)\leq a+c(k+1)\sqrt{g + p}+c\enspace.$$
Moreover, for some constant $c'>0$, for all integers $g,a\geq0$ and
$p\geq1$ and $k\geq2$, there is a graph $G$ in $\G(g,p,k,a)$ such
that $$\eta(G)\geq a+c'k\sqrt{g + p}\enspace.$$
\end{theorem}

Let $\textup{RS}(G)$ be the minimum integer $k$ such that $G$ is a
subgraph of a graph in $\G(k,k,k,k)^+$.  The graph minor structure
theorem \citep{RS-GraphMinorsXVI-JCTB03} says that $\textup{RS}(G)\leq
f(\eta(G))$ for some function $f$ independent of $G$. Conversely,
Theorem~\ref{thm:Main} implies that $\eta(G)\leq f'(\textup{RS}(G))$
for some (much smaller) function $f'$.  In this sense, $\eta$ and
$\textup{RS}$ are ``tied''. Note that such a function $f'$ is widely
understood to exist (see for instance Diestel~\cite[p.~340]{Diestel05}
and Lov\'asz~\cite{LovaszSurvey}). However, the authors are not aware
of any proof.  In addition to proving the existence of $f'$, this
paper determines the best possible function $f'$ (up to a constant
factor).

Following the presentation of definitions and other preliminary
results in Section~\ref{sec:Definitions}, the proof of the upper and
lower bounds in Theorem~\ref{thm:Main} are respectively presented in
Sections~\ref{sec:UpperBound} and \ref{sec:Constructions}.

\section{Definitions and Preliminaries}
\label{sec:Definitions}

All graphs in this paper are finite and simple, unless otherwise
stated.  Let $V(G)$ and $E(G)$ denote the vertex and edge sets of a
graph $G$. For background graph theory see \citep{Diestel05}.


A graph $H$ is a \DEF{minor} of a graph $G$ if $H$ can be obtained
from a subgraph of $G$ by contracting edges. (Note that, since we only
consider simple graphs, loops and parallel edges created during an
edge contraction are deleted.)  An \DEF{$H$-model} in $G$ is a
collection $\{S_{x}: x\in V(H)\}$ of pairwise vertex-disjoint
connected subgraphs of $G$ (called \DEF{branch sets}) such that, for
every edge $xy \in E(H)$, some edge in $G$ joins a vertex in $S_{x}$
to a vertex in $S_{y}$.  Clearly, $H$ is a minor of $G$ if and only if
$G$ contains an $H$-model. For a recent survey on graph minors see
\citep{KM-GC07}.

Let $\lex{G}{k}$ denote the \DEF{lexicographic product} of $G$ with
$K_{k}$, namely the graph obtained by replacing each vertex $v$ of $G$
with a clique $C_{v}$ of size $k$, where for each edge $vw\in E(G)$,
each vertex in $C_v$ is adjacent to each vertex in $C_w$. Let $\tw(G)$
be the treewidth of a graph $G$; see \cite{Diestel05} for background
on treewidth.



\begin{lemma}
  \label{lem:treewidth}
  For every graph $G$ and integer $k\geq 1$, every minor of
  $\lex{G}{k}$ has minimum degree at most $k \cdot \tw(G)+ k - 1$.
\end{lemma}

\begin{proof}
  A tree decomposition of $G$ can be turned into a tree decomposition
  of $\lex{G}{k}$ in the obvious way: in each bag, replace each vertex
  by its $k$ copies in $\lex{G}{k}$.  The size of each bag is
  multiplied by $k$; hence the new tree decomposition has width at
  most $k(w+1)-1=kw+k-1$, where $w$ denotes the width of the original
  decomposition.  Let $H$ be a minor of $\lex{G}{k}$. Since treewidth
  is minor-monotone,
  $$\tw(H)\leq \tw(\lex{G}{k}) \leq k\cdot \tw(G) + k - 1\enspace.$$
  The claim follows since the minimum degree of a graph is at most its
  treewidth.
\end{proof}

Note that Lemma~\ref{lem:treewidth} can be written in terms of
contraction degeneracy; see \citep{BWK06,FijWoo}.


\medskip

Let $G$ be a graph and let $\Omega = (v_{1}, v_{2}, \dots, v_{t})$ be
a circular ordering of a subset of the vertices of $G$. We write
$V(\Omega)$ for the set $\{v_{1}, v_{2}, \dots, v_{t}\}$.  A 
\DEF{circular decomposition of $G$ with perimeter $\Omega$} is a multiset
$\{C\langle w\rangle\subseteq V(G):w\in V(\Omega)\}$ of subsets of
vertices of $G$, called \DEF{bags}, that satisfy the following
properties:
\begin{itemize}
\item every vertex $w\in V(\Omega)$ is contained in its corresponding
  bag $C\langle w\rangle$;
\item for every vertex $u\in V(G) \setminus V(\Omega)$, there exists
  $w \in V(\Omega)$ such that $u$ is in $ C\langle w\rangle$;
\item for every edge $e\in E(G)$, there exists $w \in V(\Omega)$ such
  that both endpoints of $e$ are in $ C\langle w\rangle$, and
\item for each vertex $u\in V(G)$, if $u \in C\langle v_{i}\rangle,
  C\langle v_{j}\rangle$ with $i < j$ then $u \in C\langle
  v_{i+1}\rangle, \dots, C\langle v_{j-1}\rangle$ or $u \in C\langle
  v_{j+1}\rangle, \dots, C\langle v_{t}\rangle, C\langle v_{1}\rangle,
  \dots, C\langle v_{i-1}\rangle$.
\end{itemize}
(The last condition says that the bags in which $u$ appears correspond
to consecutive vertices of $\Omega$.)  The \DEF{width} of the
decomposition is the maximum cardinality of a bag minus $1$.  The
ordered pair $(G, \Omega)$ is called a \DEF{vortex}; its width is the
minimum width of a circular decomposition of $G$ with perimeter
$\Omega$.

A \DEF{surface} is a non-null compact connected 2-manifold without
boundary.  Recall that the \DEF{Euler genus} of a surface $\Sigma$ is
$2 - \chi(\Sigma)$, where $\chi(\Sigma)$ denotes the Euler
characteristic of $\Sigma$. Thus the orientable surface with $h$
handles has Euler genus $2h$, and the non-orientable surface with $c$
cross-caps has Euler genus $c$.  The boundary of an open disc $D
\subset \Sigma$ is denoted by $\bd(D)$.

See \citep{MoharThom} for basic terminology and results about graphs
embedded in surfaces. When considering a graph $G$ embedded in a
surface $\Sigma$, we use $G$ both for the corresponding abstract graph
and for the subset of $\Sigma$ corresponding to the drawing of $G$. An
embedding of $G$ in $\Sigma$ is \emph{2-cell} if every face is
homeomorphic to an open disc.

Recall Euler's formula: if an $n$-vertex $m$-edge graph is 2-cell
embedded with $f$ faces in a surface of Euler genus $g$, then
$n-m+f=2-g$. Since $2m\geq 3f$,
\begin{equation}
  \label{eqn:Euler}
  m \leq 3n + 3g - 6\enspace,
\end{equation}
which in turn implies the following well-known upper bound on the
Hadwiger number.
\begin{lemma}
  \label{lem_upper_bound_Hadwiger_surface}
  If a graph $G$ has an embedding in a surface $\Sigma$ with Euler
  genus $g$, then
$$\eta(G) \leq \sqrt{6g} + 4\enspace.$$ 
\end{lemma}
\begin{proof}
  Let $t:=\eta(G)$. Then $K_t$ has an embedding in $\Sigma$. It is
  well-known that this implies that $K_t$ has a 2-cell embedding in a
  surface of Euler genus at most $g$ (see \citep{MoharThom}). Hence
  $\binom{t}{2} \leq 3t + 3g - 6$ by \eqref{eqn:Euler}. In particular,
  $t \leq \sqrt{6g} + 4$.
\end{proof}


Let $G$ be an embedded multigraph, and let $F$ be a facial walk of
$G$. Let $v$ be a vertex of $F$ with degree more than $3$. Let
$e_1,\dots,e_d$ be the edges incident to $v$ in clockwise order around
$v$, such that $e_1$ and $e_d$ are in $F$.  Let $G'$ be the embedded
multigraph obtained from $G$ as follows.  First, introduce a path
$x_1,\dots,x_d$ of new vertices.  Then for each $i\in[1,d]$, replace
$v$ as the endpoint of $e_i$ by $x_i$.  The clockwise ordering around
$x_{i}$ is as described in
Figure~\ref{fig:VertexSplittingAtFace}. Finally delete $v$.  We say
that $G'$ is obtained from $G$ by \DEF{splitting} $v$ at $F$. Each
vertex $x_i$ is said to \emph{belong} to $v$.  By construction, $x_i$
has degree at most $3$.  Observe that there is a one-to-one
correspondence between facial walks of $G$ and $G'$. This process can
be repeated at each vertex of $F$. The embedded graph that is obtained
is called the \emph{splitting} of $G$ at $F$. And more generally, if
$F_1,\dots,F_p$ are pairwise vertex-disjoint facial walks of $G$, then
the embedded graph that is obtained by splitting each $F_i$ is called
the \emph{splitting} of $G$ at $F_1,\dots,F_p$. (Clearly, the
splitting of $G$ at $F_1,\dots,F_p$ is unique.)\

\begin{figure}[!htb]
  \begin{center}
    \includegraphics{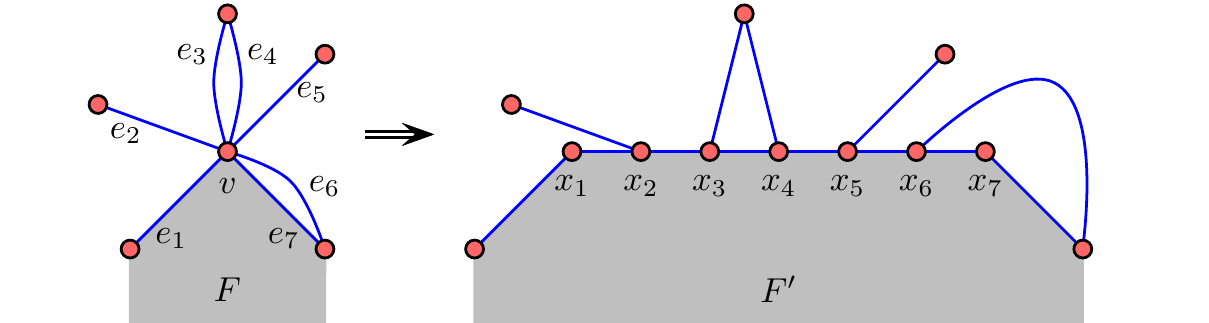}
  \end{center}
  \caption{Splitting a vertex $v$ at a face $F$.}
  \label{fig:VertexSplittingAtFace}
\end{figure}



\medskip For $g, p,k\geq0$, a graph $G$ is \emph{$(g,p,k)$-almost
  embeddable} if there exists a graph $G_{0}$ embedded in a surface
$\Sigma$ of Euler genus at most $g$, and there exist $q\leq p$
vortices $(G_{1}, \Omega_{1}), \dots, (G_{q}, \Omega_{q})$, each of
width at most $k$, such that
\begin{itemize}
\item $G = G_{0} \cup G_{1} \cup \cdots \cup G_{q}$;
\item the graphs $G_{1}, \dots, G_q$ are pairwise vertex-disjoint;
\item $V(G_i)\cap V(G_0)=V(\Omega_i)$ for all $i\in[1,q]$, and
\item there exist $q$ disjoint closed discs in $\Sigma$ whose
  interiors $D_{1}, \dots, D_{q}$ are disjoint from $G_{0}$, whose
  boundaries meet $G_{0}$ only in vertices, and such that $\bd(D_{i})
  \cap V(G_{0}) = V(\Omega_{i})$ and the cyclic ordering $\Omega_{i}$
  is compatible with the natural cyclic ordering of $V(\Omega_{i})$
  induced by $\bd(D_i)$, for all $i\in [1, q]$.
\end{itemize}

Let $\G(g,p,k)$ be the set of $(g,p,k)$-almost embeddable graphs. Note
that $\G(g,0,0)$ is exactly the class of graphs with Euler genus at
most $g$.  Also note that the literature defines a graph to be 
\DEF{$h$-almost embeddable} if it is $(h,h,h)$-almost embeddable. To
enable more accurate results we distinguish the three parameters.

Let $G_1$ and $G_2$ be disjoint graphs. Let $\{v_1,\dots,v_k\}$ and
$\{w_1,\dots,w_k\}$ be cliques of the same cardinality in $G_1$ and
$G_2$ respectively. A \DEF{clique-sum} of $G_1$ and $G_2$ is any
graph obtained from $G_1\cup G_2$ by identifying $v_i$ with $w_i$ for
each $i\in[1,k]$, and possibly deleting some of the edges $v_iv_j$.

The above definitions make precise the definition of $\G(g,p,k,a)^{+}$
given in the introduction.  We conclude this section with an easy
lemma on clique-sums.

\begin{lemma}
  \label{lem:CliqueSum}
  If a graph $G$ is a clique-sum of graphs $G_1$ and $G_2$,
  then $$\eta(G)\leq\max\{\eta(G_1),\eta(G_2)\}\enspace.$$
\end{lemma}

\begin{proof}
  Let $t:=\eta(G)$ and let $S_1,\dots,S_t$ be the branch sets of a
  $K_{t}$-model in $G$.  If some branch set $S_i$ were contained in
  $G_1 \setminus V(G_2)$, and some branch set $S_j$ were contained in
  $G_2 \setminus V(G_1)$, then there would be no edge between $S_i$
  and $S_j$ in $G$, which is a contradiction.  Thus every branch set
  intersects $V(G_1)$, or every branch set intersects $V(G_2)$.
  Suppose that every branch set intersects $V(G_1)$.  For each branch
  set $S_i$ that intersects $G_1\cap G_2$ remove from $S_{i}$ all
  vertices in $V(G_{2}) \setminus V(G_{1})$.  Since $V(G_1)\cap
  V(G_2)$ is a clique in $G_1$, the modified branch sets yield a
  $K_{t}$-model in $G_{1}$.  Hence $t\leq\eta(G_1)$.  By symmetry,
  $t\leq\eta(G_2)$ in the case that every branch set intersects $G_2$.
  Therefore $\eta(G)\leq\max\{\eta(G_1),\eta(G_2)\}$.
\end{proof}

\section{Proof of Upper Bound}
\label{sec:UpperBound}

The aim of this section is to prove the following theorem.

\begin{theorem}
  \label{thm:upper_bound}
  For all integers $g,p,k\geq0$, every graph $G$ in $\G(g,p,k)$
  satisfies
$$\eta(G)\leq 48(k+1)\sqrt{g+p} + \sqrt{6g} + 5\enspace.$$
\end{theorem}

Combining this theorem with Lemma~\ref{lem:CliqueSum} gives the
following quantitative version of the first part of
Theorem~\ref{thm:Main}.

\begin{corollary}
  \label{cor:CliqueMinor}
  For every graph $G \in \G(g,p,k,a)^+$,
$$\eta(G) \leq a + 48(k+1)\sqrt{g+p} + \sqrt{6g} + 5\enspace.$$
\end{corollary}
\begin{proof}
  Let $G \in \G(g,p,k,a)^+$.  Lemma~\ref{lem:CliqueSum} implies that
  $\eta(G) \leq \eta(G')$ for some graph $G' \in \G(g,p,k,a)$.
  Clearly, $\eta(G') \leq \eta(G' \setminus A) + a$, where $A$ denotes
  the (possibly empty) apex set of $G'$.  Since $G' \setminus A \in
  \G(g,p,k)$, the claim follows from Theorem~\ref{thm:upper_bound}.
\end{proof}

The proof of Theorem~\ref{thm:upper_bound} uses the following
definitions. Two subgraphs $A$ and $B$ of a graph $G$ \DEF{touch} if
$A$ and $B$ have at least one vertex in common or if there is an edge
in $G$ between a vertex in $A$ and another vertex in $B$.  We
generalize the notion of minors and models as follows. For an integer
$k \geq 1$, a graph $H$ is said to be an \DEF{$(H,k)$-minor} of a
graph $G$ if there exists a collection $\{S_{x}: x \in V(H)\}$ of
connected subgraphs of $G$ (called \emph{branch sets}), such that
$S_{x}$ and $S_{y}$ touch in $G$ for every edge $xy \in E(H)$, and
every vertex of $G$ is included in at most $k$ branch sets in the
collection.  The collection $\{S_{x}:x\in V(H)\}$ is called an 
\DEF{$(H,k)$-model} in $G$.  Note that for $k=1$ this definition
corresponds to the usual notions of $H$-minor and $H$-model.  As shown
in the next lemma, this generalization provides another way of
considering $H$-minors in $\lex{G}{k}$, the lexicographic product of
$G$ with $K_{k}$. (The easy proof is left to the reader.)

\begin{lemma}
  \label{lem:lex_prod_k_minor}
  Let $k \geq 1$. A graph $H$ is an $(H, k)$-minor of a graph $G$ if
  and only if $H$ is a minor of $\lex{G}{k}$.
\end{lemma}

For a surface $\Sigma$, let $\Sigma_{c}$ be $\Sigma$ with $c$ cuffs
added; that is, $\Sigma_c$ is obtained from $\Sigma$ by removing the
interior of $c$ pairwise disjoint closed discs. (It is well-known that the
locations of the discs are irrelevant.)
When considering graphs embedded in $\Sigma_{c}$ we require the
embedding to be $2$-cell.  We emphasize that this is a non-standard
and relatively strong requirement; in particular, it implies that the
graph is connected, and the boundary of each cuff intersects the graph 
in a cycle. Such cycles are called \DEF{cuff-cycles}.

For $g\geq 0$ and $c\geq 1$, a graph $G$ is \DEF{$(g, c)$-embedded}
if $G$ has maximum degree $\Delta(G) \leq 3$ and $G$ is embedded in a
surface of Euler genus at most $g$ with at most $c$ cuffs added, such
that {\em every} vertex of $G$ lies on the boundary of the surface.
(Thus the cuff-cycles induce a partition of the whole vertex set.)\
The graph $G$ is \DEF{$(g, c)$-embeddable} if there exists such an
embedding.  Note that if $C$ is a contractible cycle in a
$(g,c)$-embedded graph, then the closed disc bounded by $C$ is
uniquely determined even if the underlying surface is the sphere
(because there is at least one cuff).


\begin{lemma}
  \label{lem:combined_lemma}
  For every graph $G \in \G(g,p,k)$ there exists a $(g,p)$-embeddable
  graph $H$ with $\eta(G) \leq \eta(\lex{H}{k+1})+\sqrt{6g} + 4$.
\end{lemma}

\begin{proof}
  Let $t:= \eta(G)$. Let $S_1,\dots,S_t$ be the branch sets of a
  $K_t$-model in $G$. Since $\eta(G)$ equals the Hadwiger number of
  some connected component of $G$, we may assume that $G$ is
  connected.  Thus we may `grow' the branch sets until
  $V(S_1)\cup\dots\cup V(S_t)=V(G)$.


  Write $G = G_{0} \cup G_{1}\cup\cdots\cup G_{q}$ as in the
  definition of $(g,p,k)$-almost embeddable graphs. Thus $G_{0}$ is
  embedded in a surface $\Sigma$ of Euler genus at most $g$, and
  $(G_{1}, \Omega_{1}), \dots, (G_{q},\Omega_{q})$ are pairwise vertex-disjoint
  vortices of width at most $k$, for some $q\leq p$.  Let $D_{1},
  \dots, D_{q}$ be the proper interiors of the closed discs of
  $\Sigma$ appearing in the definition.

  Define $r$ and reorder the branch sets, so that each $S_i$ contains
  a vertex of some vortex if and only if $i \leq r$.  If $t > r$, then
  $S_{r+1},\dots, S_{t}$ is a $K_{t-r}$-model in the embedded graph
  $G_{0}$, and hence $t-r \leq \sqrt{6g} + 4$ by
  Lemma~\ref{lem_upper_bound_Hadwiger_surface}.  Therefore, it
  suffices to show that $r \leq \eta(\lex{H}{k+1})$ for some
  $(g,p)$-embeddable graph $H$.

  Modify $G$, $G_{0}$, and the branch sets $S_{1}, \dots, S_{r}$ as
  follows.
  First, remove from $G$ and $G_{0}$ every vertex of $S_{i}$ for all
  $i\in [r+1,t]$.
  Next, while some branch set $S_{i}$ ($i\in[1,r]$) contains an edge
  $uv$ in $G_{0}$ where $u$ is in some vortex, but $v$ is in no
  vortex, contract the edge $uv$ into $u$ (this operation is done in
  $S_{i}$, $G$, and $G_{0}$).  The above operations on $G_{0}$ are
  carried out in its embedding in the natural way. 
Now apply a final operation on $G$
  and $G_{0}$: for each $j\in [1,q]$ and each pair of consecutive
  vertices $a$ and $b$ in $\Omega_{j}$, remove the edge $ab$ if it
  exists, and embed the edge $ab$ as a curve on the boundary of $D_j$.

  When the above procedure is finished, every vertex of the modified
  $G_{0}$ belongs to some vortex.  It should be clear that the
  modified branch sets $S_{1},\dots,S_{r}$ still provide a model of
  $K_{r}$ in $G$.  Also observe that $G_{0}$ is connected; this is
  because $V(\Omega_{j})$ induces a connected subgraph for each $j\in
  [1,q]$, and each vortex intersects at least one branch set $S_{i}$
  with $i\in[1,r]$.  By the final operation, the boundary of the disc
  $D_{j}$ of $\Sigma$ intersects $G_{0}$ in a cycle $C_{j}$ of $G_{0}$
  with $V(C_{j}) = V(\Omega_{j})$ and such that $C_{j}$ (with the
  right orientation) defines the same cyclic ordering as $\Omega_{j}$
  for every $j\in[1,q]$.

  We claim that $G_{0}$ can be $2$-cell embedded in a surface
  $\Sigma'$ with Euler genus at most that of $\Sigma$, such that each
  $C_{j}$ ($j \in [1,q]$) is a facial cycle of the embedding. This
  follows by considering the combinatorial embedding (that is,
  circular ordering of edges incident to each vertex, and edge
  signatures) determined by the embedding in $\Sigma$
  (see~\cite{MoharThom}), and observing that under the above
  operations, the Euler genus of the combinatorial embedding does not
  increase, and facial walks remain facial walks (so that each $C_j$
  is a facial cycle).  Now, removing the $q$ open discs
  corresponding to these facial cycles gives a $2$-cell embedding
  of $G_0$ in $\Sigma'_q$.


  We now prove that $\eta(G_{0}[k+1])\geq r$.  For every $i\in [1,q]$,
  let $\{C\ang{w}\subseteq V(G_{i}):w\in V(\Omega_{i})\}$ denote a
  circular decomposition of width at most $k$ of the $i$-th vortex.
  For each $i\in [1,r]$, mark the vertices $w$ of $G_{0}$ for which
  $S_{i}$ contains at least one vertex in the bag $C\langle w\rangle$
  (recall that every vertex of $G_{0}$ is in the perimeter of some
  vortex), and define $S'_{i}$ as the subgraph of $G_{0}$ induced by
  the marked vertices. It is easily checked that $S'_{i}$ is a
  connected subgraph of $G_{0}$.  Also, $S'_{j}$ and $S'_{i}$ touch in
  $G_{0}$ for all $i \neq j$.  Finally, a vertex of $G_{0}$ will be
  marked at most $k+1$ times, since each bag has size at most $k+1$.
  It follows that $\{S'_{1}, \dots, S'_{r}\}$ is a $(K_{r},
  k+1)$-model in $G_{0}$, which implies by
  Lemma~\ref{lem:lex_prod_k_minor} that $K_{r}$ is minor of
  $G_{0}[k+1]$, as claimed.


  Finally, let $H$ be obtained from $G_{0}$ by splitting each vertex
  $v$ of degree more than $3$ along the cuff boundary that contains
  $v$.  (Clearly the notion of splitting along a face extends to
  splitting along a cuff.)\ By construction, $\Delta(H) \leq 3$ and
  $H$ is $(g,q)$-embedded.  The $(K_{r}, k+1)$-model of $G_{0}$
  constructed above can be turned into a $(K_{r}, k+1)$-model of $H$
  by replacing each branch set $S'_i$ by the union, taken over the
  vertices $v\in V(S'_i)$, of the set of vertices in $H$ that belong
  to $v$.  Hence $r \leq \eta(G_{0}[k+1]) \leq \eta(\lex{H}{k+1})$.
\end{proof}

We need to introduce a few definitions.  Consider a $(g,c)$-embedded
graph $G$.  An edge $e$ of $G$ is said to be a \DEF{cuff} or a 
\DEF{non-cuff} edge, depending on whether $e$ is included in a
cuff-cycle. Every non-cuff edge has its two endpoints in either the
same cuff-cycle or in two distinct cuff-cycles.  Since $\Delta(G)\leq
3$, the set of non-cuff edges is a matching.

A cycle $C$ of $G$ is an \emph{$F$-cycle} where $F$ is the set of
non-cuff edges in $C$.
A non-cuff edge $e$ is \DEF{contractible} if there exists a
contractible $\{e\}$-cycle, and is \DEF{noncontractible} otherwise.
Two non-cuff edges $e$ and $f$ are \DEF{homotopic} if $G$ contains a
contractible $\{e,f\}$-cycle.  Observe that if $e$ and $f$ are
homotopic, then they have their endpoints in the same cuff-cycle(s),
as illustrated in Figure~\ref{fig:HomotopicEdges}. We now prove that
homotopy defines an equivalence relation on the set of noncontractible
non-cuff edges of $G$.

\begin{figure}
  \centering
  \includegraphics{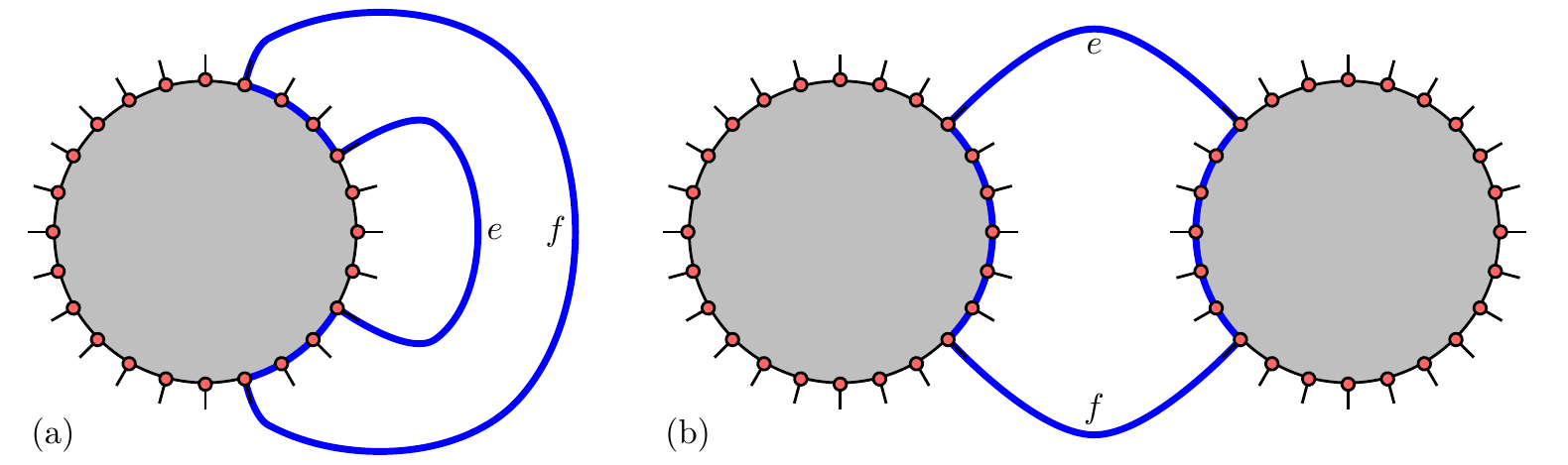}
  \caption{\label{fig:HomotopicEdges} Homotopic edges: (a) one cuff,
    (b) two cuffs.}
\end{figure}


\begin{lemma}
  \label{lem:equivalence}
  Let $G$ be a $(g,c)$-embedded graph, and let $e_{1}, e_{2}, e_{3}$
  be distinct noncontractible non-cuff edges of $G$, such that $e_{1}$
  is homotopic to $e_{2}$ and to $e_{3}$.  Then $e_{2}$ and $e_{3}$
  are also homotopic. Moreover, given a contractible
  $\{e_{1},e_{2}\}$-cycle $C_{12}$ bounding a closed disc $D_{12}$,
  for some distinct $i,j\in\{1,2,3\}$, there is a contractible
  $\{e_{i}, e_{j}\}$-cycle bounding a closed disc containing $e_{1},
  e_{2}, e_{3}$ and all noncontractible non-cuff edges of $G$
  contained in $D_{12}$.
\end{lemma}

\begin{proof}
  Let $C_{13}$ be a contractible $\{e_{1},e_{3}\}$-cycle. 
Let $P_{12},Q_{12}$ be the two paths in the graph $C_{12}
  \setminus\{e_{1},e_{2}\}$.  
Let $P_{13},
  Q_{13}$ be the two paths in the graph $C_{13}
  \setminus\{e_{1},e_{3}\}$.  Exchanging $P_{13}$ and $Q_{13}$ if
  necessary, we may denote the endpoints of $e_{i}$ ($i=1,2,3$) by
  $u_{i}, v_{i}$ so that the endpoints of $P_{12}$ and $P_{13}$ are
  $u_{1}, u_{2}$ and $u_{1}, u_{3}$, respectively, and similarly, the
  endpoints of $Q_{12}$ and $Q_{13}$ are $v_{1}, v_{2}$ and $v_{1},
  v_{3}$, respectively.

  Let $D_{13}$ be the closed disc bounded by $C_{13}$.  Each edge of
  $P_{1i}$ and $Q_{1i}$ ($i=2,3$) is on the boundaries of both
  $D_{1i}$ and a cuff; it follows that every non-cuff edge of $G$
  incident to an internal vertex of $P_{1i}$ or $Q_{1i}$ is entirely
  contained in the disc $D_{1i}$.  The paths $P_{12}$ and $P_{13}$ are
  subgraphs of a common cuff-cycle $C_{P}$, and $Q_{12}$ and $Q_{13}$
  are subgraphs of a common cuff-cycle $C_{Q}$.  Note that these two
  cuff-cycles could be the same.

  Recall that non-cuff edges of $G$ are independent (that is, have no
  endpoint in common). This will be used in the arguments below.  We
  claim that
  \begin{equation}
    \label{claim:disc_noncontractible}
    \begin{minipage}{0.92\textwidth}
      \textrm{every noncontractible non-cuff edge $f$ contained in
        $D_{1i}$ has one endpoint in $P_{1i}$ and the other in
        $Q_{1i}$, for each $i\in \{2,3\}$}.
    \end{minipage}
  \end{equation}
  The claim is immediate if $f\in\{e_1,e_i\}$. Now assume that
  $f\not\in\{e_{1},e _{i}\}$. The edge $f$ is incident to at least one
  of $P_{1i}$ and $Q_{1i}$ since there is no vertex in the proper
  interior of $D_{1i}$.  Without loss of generality, $f$ is incident
  to $P_{1i}$.  The edge $f$ can only be incident to internal vertices
  of $P_{1i}$, since $f$ is independent of $e_{1}$ and $e_{i}$.  Say
  $f=xy$. If $x,y\in V(P_{1i})$ then the $\{f\}$-cycle obtained by
  combining the $x$--$y$ subpath of $P_{1i}$ with the edge $f$ is
  contained in $D_{1i}$ and thus is contractible. Hence $f$ is a
  contractible non-cuff edge, a contradiction.  This proves
  \eqref{claim:disc_noncontractible}.

  First we prove the lemma in the case where $e_{3}$ is incident to
  $P_{12}$.  Since $e_{3}$ is incident to an internal vertex of
  $P_{12}$, it follows that $e_{3}$ is contained in $D_{12}$.  This
  shows the second part of the lemma. To show that $e_{2}$ and $e_{3}$
  are homotopic, consider the endpoint $v_{3}$ of $e_{3}$.  Since
  $e_{3}$ is in $D_{12}$ and $u_{3} \in V(P_{12})$, we have $v_{3} \in
  V(Q_{12})$ by \eqref{claim:disc_noncontractible}.  Now, combining
  the $u_{2}$--$u_{3}$ subpath of $P_{12}$ and the $v_{2}$--$v_{3}$
  subpath of $Q_{12}$ with $e_{2}$ and $e_{3}$, we obtain an
  $\{e_{2},e_{3}\}$-cycle contained in $D_{12}$, which is thus
  contractible.  This shows that $e_{2}$ and $e_{3}$ are homotopic.

  By symmetry, the above argument also handles the case where $e_{3}$
  is incident to $Q_{12}$.  Thus we may assume that $e_{3}$ is
  incident to neither $P_{12}$ nor $Q_{12}$.

  Suppose $P_{12} \subseteq P_{13}$.  Then, by
  \eqref{claim:disc_noncontractible}, all noncontractible non-cuff
  edges contained in $D_{12}$ are incident to $P_{12}$, and thus also
  to $P_{13}$. Hence they are all contained in the disc $D_{13}$.
  Moreover, a contractible $\{e_{2}, e_{3}\}$-cycle can be found in
  the obvious way.  Therefore the lemma holds in this case.  Using
  symmetry, the same argument can be used if $P_{12} \subseteq
  Q_{13}$, $Q_{12} \subseteq P_{13}$, or $Q_{12} \subseteq Q_{13}$.
  Thus we may assume
  \begin{equation}
    \label{eq:1213}
    P_{12} \not \subseteq P_{13}; \quad P_{12} \not \subseteq Q_{13};
    \quad Q_{12} \not \subseteq P_{13}; \quad Q_{12} \not \subseteq
    Q_{13}.
  \end{equation}

  Next consider $P_{12}$ and $P_{13}$.  If we orient these paths
  starting at $u_{1}$, then they either go in the same direction
  around $C_{P}$, or in opposite directions.  Suppose the former. Then
  one path is a subpath of the other. Since by our assumption $u_{3}$
  is not in $P_{12}$, we have $P_{12}\subseteq P_{13}$, which
  contradicts \eqref{eq:1213}.  Hence the paths $P_{12}$ and $P_{13}$
  go in opposite directions around $C_{P}$.  If $V(P_{12}) \cap
  V(P_{13}) \neq \{u_{1}\}$, then $u_{3}$ is an internal vertex of
  $P_{12}$, which contradicts our assumption on $e_{3}$.  Hence
  \begin{equation}
    \label{eq:homotopic1}
    V(P_{12}) \cap V(P_{13}) = \{u_{1}\}.
  \end{equation}

  By symmetry, the above argument show that $Q_{12}$ and $Q_{13}$ go
  in opposite directions around $C_{Q}$ (starting from $v_{1}$), which
  similarly implies
  \begin{equation}
    \label{eq:homotopic2}
    V(Q_{12}) \cap V(Q_{13}) = \{v_{1}\}.
  \end{equation}

  Now consider $P_{12}$ and $Q_{13}$. These two paths do not share any
  endpoint.  If $C_{P} \neq C_{Q}$ then obviously the two paths are
  vertex-disjoint.  If $C_{P} = C_{Q}$ and $V(P_{12})\cap V(Q_{13})
  \neq \emptyset$, then at least one of $v_{1}$ and $v_{3}$ is an
  internal vertex of $P_{12}$, because otherwise $P_{12} \subseteq
  Q_{13}$, which contradicts \eqref{eq:1213}.  However $v_{1} \notin
  V(P_{12})$ since $v_{1} \in V(Q_{12})$, and $v_{3} \notin V(P_{12})$
  by our assumption that $e_{3}$ is not incident to $P_{12}$.  Hence,
  in all cases,
  \begin{equation}
    \label{eq:homotopic3}
    V(P_{12}) \cap V(Q_{13}) = \emptyset.
  \end{equation}
  By symmetry,
  \begin{equation}
    \label{eq:homotopic4}
    V(Q_{12}) \cap V(P_{13}) = \emptyset.
  \end{equation}

  It follows from \eqref{eq:homotopic1}--\eqref{eq:homotopic4} that
  $C_{12}$ and $C_{13}$ only have $e_{1}$ in common. This implies in
  turn that $D_{12}$ and $D_{13}$ have disjoint proper interiors.
  Thus the cycle $C_{23} := (C_{12} \cup C_{13}) - e_{1}$ bounds the
  disc obtained by gluing $D_{12}$ and $D_{13}$ along $e_{1}$. Hence
  $C_{23}$ is an $\{e_{2}, e_{3}\}$-cycle of $G$ bounding a disc
  containing $e_{3}$ and all edges contained in $D_{12}$. This
  concludes the proof.
\end{proof}


The next lemma is a direct consequence of Lemma~\ref{lem:equivalence}.
An equivalence class $\mathcal{Q}$ for the homotopy relation on the
noncontractible non-cuff edges of $G$ is \DEF{trivial} if
$|\mathcal{Q}|=1$, and \DEF{non-trivial} otherwise.

\begin{lemma}
  \label{lem:disc}
  Let $G$ be a $(g,c)$-embedded graph and let $\mathcal{Q}$ be a
  non-trivial equivalence class of the noncontractible non-cuff edges
  of $G$.  Then there are distinct edges $e,f\in\mathcal{Q}$ and a
  contractible $\{e,f\}$-cycle $C$ of $G$, such that the closed disc
  bounded by $C$ contains every edge in $\mathcal{Q}$.
\end{lemma}

Our main tool in proving Theorem~\ref{thm:upper_bound} is the
following lemma, whose inductive proof is enabled by the following
definition. Let $G$ be a $(g,c)$-embedded graph and let $k\geq 1$.  A
graph $H$ is a \DEF{\wminor{k}} of $G$ if there exists an
$(H,4k)$-model $\{S_{x}: x\in V(H)\}$ in $G$ such that, for every
vertex $u\in V(G)$ incident to a noncontractible non-cuff edge in a
non-trivial equivalence class, the number of subgraphs in the model
including $u$ is at most $k$.  Such a collection $\{S_{x}: x\in
V(H)\}$ is said to be a \DEF{\wmodel{k}} of $H$ in $G$.  This
provides a relaxation of the notion of $(H,k)$-minor since some
vertices of $G$ could appear in up to $4k$ branch sets (instead of
$k$).  We emphasize that this definition depends heavily on the
embedding of $G$.


\begin{lemma}
  \label{lem:main_tool}
  Let $G$ be a $(g,c)$-embedded graph and let $k\geq 1$. Then every
  {\wminor{k}} $H$ of $G$ has minimum degree at most $\f{k}{c}{g}$.
\end{lemma}

\begin{proof}
  Let $q(G)$ be the number of non-trivial equivalence classes of
  noncontractible non-cuff edges in $G$.  We proceed by induction,
  firstly on $g+c$, then on $q(G)$, and then on $|V(G)|$.  Now $G$ is
  embedded in a surface of Euler genus $g' \leq g$ with $c'\leq c$
  cuffs added. If $g'<g$ or $c'<c$ then we are done by induction. Now
  assume that $g'=g$ and $c'=c$.


  We repeatedly use the following observation: If $C$ is a
  contractible cycle of $G$, then the subgraph of $G$ consisting of
  the vertices and edges contained in the closed disc $D$ bounded by
  $C$ is outerplanar, and thus has treewidth at most 2.  This is
  because the proper interior of $D$ contains no vertex of $G$ (since
  all the vertices in $G$ are on the cuff boundaries).

  Let $\{S_{x}: x\in V(H)\}$ be a {\wmodel{k}} of $H$ in $G$.  Let $d$
  be the minimum degree of $H$.  We may assume that $d\geq 20k$, as
  otherwise $d\leq 48k\sqrt{c+g}$   (since $c\geq 1$)   and we are
  done. Also, it is enough to prove the lemma when 
$H$ is connected, so assume this is the case.


  \case{1}{Some non-cuff edge $e$ of $G$ is contractible.}  Let $C$ be
  a contractible $\{e\}$-cycle. Let $u,v$ be the endpoints of $e$.
  Remove from $G$ every vertex in $V(C) \setminus \{u,v\}$ and modify
  the embedding of $G$ by redrawing the edge $e$ where the path $C-e$
  was. Thus $e$ becomes a cuff-edge in the resulting graph $G'$, and
  $u$ and $v$ both have degree $2$. Also observe that $G'$ is
  connected and remains simple (that is, this operation does not
  create loops or parallel edges).  Since the embedding of $G'$ is
  $2$-cell, $G'$ is $(g,c)$-embedded also.

  If $e_{1}$ and $e_{2}$ are noncontractible non-cuff edges of $G'$
  that are homotopic in $G'$, then $e_{1}$ and $e_{2}$ were also
  noncontractible and homotopic in $G$. Hence, $q(G') \leq q(G)$.
  Also, $|V(G')| < |V(G)|$ since we removed at least one vertex from
  $G$.  Thus, by induction, every {\wminor{k}} of $G'$ has minimum
  degree at most $\f{k}{c}{g}$. Therefore, it is enough to show that
  $H$ is also a {\wminor{k}} of $G'$.

  Let $G_{1}$ be the subgraph of $G$ lying in the closed disc bounded
  by $C$; observe that $G_{1}$ is outerplanar. Moreover, $(G_{1},G')$
  is a separation of $G$ with $V(G_{1}) \cap V(G') = \{u,v\}$. (That
  is, $G_{1} \cup G' = G$ and $V(G_{1}) \setminus V(G')\neq\emptyset$
  and $V(G') \setminus V(G_{1}) \neq\emptyset$.)

  First suppose that $S_{x} \subseteq G_{1} \setminus \{u,v\}$ for
  some vertex $x\in V(H)$.  Let $H'$ be the subgraph of $H$ induced by
  the set of such vertices $x$.  In $H$, the only neighbors of a
  vertex $x\in V(H')$ that are not in $H'$ are vertices $y$ such that
  $S_{y}$ includes at least one of $u, v$.  There are at most $2\cdot
  4k = 8k$ such branch sets $S_{y}$. Hence, $H'$ has minimum degree at
  least $d - 8k \geq 12k$.  However, $H'$ is a minor of $G_{1}[4k]$
  and hence has minimum degree at most $4k\cdot\tw(G_{1}) + 4k - 1
  \leq 12k - 1$ by Lemma~\ref{lem:treewidth}, a contradiction.

  It follows that every branch set $S_{x}$ ($x\in V(H)$) contains at
  least one vertex in $V(G')$.  Let $S'_{i} := S_{i} \cap G'$. Using
  the fact that $uv \in E(G')$, it is easily seen that the collection
  $\{S'_{x}: x \in V(H)\}$ is a {\wmodel{k}} of $H$ in $G'$.

  \case{2}{Some equivalence class $\mathcal{Q}$ is non-trivial.}  By
  Lemma~\ref{lem:disc}, there are two edges $e,f \in\mathcal{Q}$ and a
  contractible $\{e,f\}$-cycle $C$ such that every edge in
  $\mathcal{Q}$ is contained in the disc bounded by $C$.  Let $P_{1},
  P_{2}$ be the two components of $C \setminus \{e,f\}$.  These two
  paths either belong to the same cuff-cycle or to two distinct
  cuff-cycles of $G$.

  Our aim is to eventually contract each of $P_{1}, P_{2}$ into a
  single vertex. However, before doing so we slightly modify $G$ as
  follows. For each cuff-cycle $C^{*}$ intersecting $C$, select an
  arbitrary edge in $E(C^{*}) \setminus E(C)$ and subdivide it {\em
    twice}.  Let $G'$ be the resulting $(g,c)$-embedded graph.
  Clearly $q(G') = q(G)$, and there is an obvious {\wmodel{k}}
  $\{S'_{x}:x\in V(H)\}$ of $H$ in $G'$: simply apply the same
  subdivision operation on the branch sets $S_{x}$.

  Let $G'_{1}$ be the subgraph of $G'$ lying in the closed disc $D$
  bounded by $C$.  Observe that $G'_{1}$ is outerplanar with
  outercycle $C$.  Suppose that some edge $xy$ in $E(G'_{1}) \setminus
  E(C)$ has both its endpoints in the same path $P_{i}$, for some
  $i\in \{1,2\}$.  Then the cycle obtained by combining $xy$ and the
  $x$--$y$ path in $P_{i}$ is a contractible cycle of $G'$, and its
  only non-cuff edge is $xy$.  The edge $xy$ is thus a contractible
  edge of $G'$, and hence also of $G$, a contradiction.

  It follows that every non-cuff edge included in $G'_{1}$ has one
  endpoint in $P_{1}$ and the other in $P_{2}$. Hence, every such edge
  is homotopic to $e$ and therefore belongs to $\mathcal{Q}$.

  Consider the {\wmodel{k}} $\{S'_{x}:x\in V(H)\}$ of $H$ in $G'$
  mentioned above.  Let $e=uv$ and $f=u'v'$, with $u, u' \in V(P_{1})$
  and $v, v' \in V(P_{2})$.  Let $X := \{u, u', v, v'\}$. For each
  $w\in X$, the number of branch sets $S'_{x}$ that include $w$ is at
  most $k$, since $e$ and $f$ are homotopic noncontractible non-cuff
  edges.

  Let $J := G'_{1} \setminus X$.  Note that $\tw(J) \leq 2$ since
  $G'_1$ is outerplanar. Let $Z := \{x \in V(H): S'_{x} \subseteq
  J\}$.  First, suppose that $Z\neq\emptyset$.  Every vertex of $J$ is
  in at most $4k$ branch sets $S'_{x}$ ($x\in Z$).  It follows that
  the induced subgraph $H[Z]$ is a minor of $\lex{J}{4k}$.  Thus, by
  Lemma~\ref{lem:treewidth}, $H[Z]$ has a vertex $y$ with degree at
  most $4k \cdot \tw(J) + 4k - 1 \leq 4k \cdot 2 + 4k - 1 =12k - 1$.
  Consider the neighbors of $y$ in $H$.  Since $X$ is a cutset of $G'$
  separating $V(J)$ from $G' \setminus V(G'_{1})$, the only neighbors
  of $y$ in $H$ that are not in $H[Z]$ are vertices $x$ such that
  $V(S'_{x}) \cap X \neq \emptyset$. As mentioned before, there are at
  most $4k$ such vertices; hence, $y$ has degree at most $12k - 1 + 4k
  = 16k - 1$. However this contradicts the assumption that $H$ has
  minimum degree $d \geq 20k$.  Therefore, we may assume that
  $Z=\emptyset$; that is, every branch set $S'_{x}$ ($x \in V(H)$)
  intersecting $V(G'_{1})$ contains some vertex in $X$.

  Now, remove from $G'$ every edge in $\mathcal{Q}$ except $e$, and
  contract each of $P_{1}$ and $P_{2}$ into a single vertex. Ensuring
  that the contractions are done along the boundary of the relevant
  cuffs in the embedding.  This results in a graph $G''$ which is
  again $(g,c)$-embedded. Note that $G''$ is guaranteed to be simple,
  thanks to the edge subdivision operation that was applied to $G$
  when defining $G'$.

  If a non-cuff edge is contractible in $G''$ then it is also
  contractible in $G'$, implying all non-cuff edges in $G''$ are
  noncontractible. Two non-cuff edges of $G''$ are homotopic in $G''$
  if and only if they are in $G'$.  It follows $q(G'') = q(G') - 1 =
  q(G)-1$, since $e$ is not homotopic to another non-cuff edge in
  $G''$.  By induction, every {\wminor{k}} of $G''$ has minimum degree
  at most $\f{k}{c}{g}$. Thus, it suffices to show that $H$ is also a
  {\wminor{k}} of $G''$.

  For $x \in V(H)$, let $S''_{x}$ be obtained from $S'_{x}$ by
  performing the same contraction operation as when defining $G''$
  from $G'$: every edge in $\mathcal{Q}\setminus\{e\}$ is removed and
  every edge in $E(P_{1}) \cup E(P_{2})$ is contracted.  Using that
  every subgraph $S'_{x}$ either is disjoint from $V(G'_{1})$ or
  contains some vertex in $X$, it can be checked that $S''_{x}$ is
  connected.

  Consider an edge $xy \in E(H)$. We now show that the two subgraphs
  $S''_{x}$ and $S''_{y}$ touch in $G''$. Suppose $S'_{x}$ and
  $S'_{y}$ share a common vertex $w$.  If $w \notin V(G'_{1})$, then
  $w$ is trivially included in both $S''_{x}$ and $S''_{y}$.  If $w
  \in V(G'_{1})$, then each of $S'_{x}$ and $S'_{y}$ contain a vertex
  from $X$, and hence either $u$ or $v$ is included in both $S''_{x}$
  and $S''_{y}$, or $u$ is included in one and $v$ in the other. In
  the latter case $uv$ is an edge of $G''$ joining $S''_{x}$ and
  $S''_{y}$.  Now assume $S'_{x}$ and $S'_{y}$ are
  vertex-disjoint. Thus there is an edge $ww' \in E(G')$ joining these
  two subgraphs in $G'$. Again, if neither $w$ nor $w'$ belong to
  $V(G'_{1})$, then obviously $ww'$ joins $S''_{x}$ and $S''_{y}$ in
  $G''$. If $w, w' \in V(G'_{1})$, then each of $S'_{x}$ and $S'_{y}$
  contain a vertex from $X$, and we are done exactly as previously.
  If exactly one of $w, w'$ belongs to $V(G'_{1})$, say $w$, then $w
  \in X$ and $w'$ is the unique neighbor of $w$ in $G'$ outside
  $V(G'_{1})$. The contraction operation naturally maps $w$ to a
  vertex $m(w) \in \{u, v\}$. The edge $w'm(w)$ is included in $G''$
  and thus joins $S''_{x}$ and $S''_{y}$.

  In order to conclude that $\{S''_{x}:x\in V(H)\}$ is a {\wmodel{k}}
  of $H$ in $G''$, it remains to show that, for every vertex $w\in
  V(G'')$, the number of branch sets including $w$ is at most $4k$,
  and is at most $k$ if $w$ is incident to a non-cuff edge homotopic
  to another non-cuff edge.  This condition is satisfied if $w \notin
  \{u, v\}$, because two non-cuff edges of $G''$ are homotopic in
  $G''$ if and only if they are in $G'$. Thus assume $w \in \{u, v\}$.
  By the definition of $G''$, the edge $e=uv$ is {\em not} homotopic
  to another non-cuff edge of $G''$. Moreover, for each $z \in X$,
  there are at most $k$ branch sets $S'_{x}$ ($x\in V(H)$) containing
  $z$. Since $|X|=4$, it follows that there are at most $4k$ branch
  sets $S''_{x}$ ($x\in V(H)$) containing $w$. Therefore, the
  condition holds also for $w$, and $H$ is a {\wminor{k}} of $G''$.

  \case{3}{There is at most one non-cuff edge.}  Because $G$ is
  connected, this implies that $G$ consists either of a unique
  cuff-cycle, or of two cuff-cycles joined by a non-cuff edge.  In
  both cases, $G$ has treewidth exactly $2$.  Since $H$ is a minor of
  $G[4k]$, Lemma~\ref{lem:treewidth} implies that $H$ has minimum
  degree at most $4k \cdot \tw(G) + 4k - 1 = 12k - 1\leq \f{k}{c}{g}$,
  as desired.

  \case{4}{Some cuff-cycle $C$ contains three consecutive degree-$2$
    vertices.}  Let $u,v,w$ be three such vertices (in order). Note
  that $C$ has at least four vertices, as otherwise $G=C$ and the
  previous case would apply. It follows $uw \notin E(G)$.  Let $G'$ be
  obtained from $G$ by contracting the edge $uv$ into the vertex $u$.
  In the embedding of $G'$, the edge $uw$ is drawn where the path
  $uvw$ was; thus $uw$ is a cuff-edge, and $G'$ is $(g, c)$-embedded.
  We have $q(G') = q(G)$ and $|V(G')| < |V(G)|$, hence by induction,
  $G'$ satisfies the lemma, and it is enough to show that $H$ is a
  {\wminor{k}} of $G'$.

  Consider the {\wmodel{k}} $\{S_{x}: x\in V(H)\}$ of $H$ in $G$.  If
  $V(S_{x})=\{v\}$ for some $x \in V(H)$, then $x$ has degree at most
  $3\cdot 4k - 1 = 12k - 1$ in $H$, because $xy \in E(H)$ implies that
  $S_{y}$ contains at least one of $u, v, w$.  However this
  contradicts the assumption that $H$ has minimum degree $d \geq 20k$.
  Thus every branch set $S_{x}$ that includes $v$ also contains at
  least one of $u,w$ (since $S_{x}$ is connected).

  For $x\in V(H)$, let $S'_{x}$ be obtained from $S_{x}$ as expected:
  contract the edge $uv$ if $uv \in E(S_{x})$. Clearly $S'_{x}$ is
  connected.  Consider an edge $xy \in E(H)$. If $S_{x}$ and $S_{y}$
  had a common vertex then so do $S'_{x}$ and $S'_{y}$. If $S_{x}$ and
  $S_{y}$ were joined by an edge $e$, then either $e$ is still in $G'$
  and joins $S'_{x}$ and $S'_{y}$, or $e=uv$ and $u \in V(S'_{x}),
  V(S'_{y})$. Hence in each case $S'_{x}$ and $S'_{y}$ touch in $G'$.
  Finally, it is clear that $\{S'_{x}: x\in V(H)\}$ meets remaining 
requirements to be a {\wmodel{k}} of $H$ in $G'$, since
  $V(S'_{x}) \subseteq V(S_{x})$ for every $x\in V(H)$ and the
  homotopy properties of the non-cuff edges have not changed.
  Therefore, $H$ is a {\wminor{k}} of $G'$.

  \case{5}{None of the previous cases apply.}  Let $t$ be the number
  of non-cuff edges in $G$ (thus $t\geq 2$).  Since there are no three
  consecutive degree-$2$ vertices, every cuff edge is at distance at
  most $1$ from a non-cuff edge.  It follows that
  \begin{equation}
    \label{eq:total_vs_non-cuff}
    |E(G)| \leq 9t.
  \end{equation}
  (This inequality can be improved but is good enough for our
  purposes.)

  For a facial walk $F$ of the embedded graph $G$, let $\nc(F)$ denote
  the number of occurrences of non-cuff edges in $F$. (A non-cuff edge
  that appears twice in $F$ is counted twice.)\ We claim that $\nc(F)
  \geq 3$.  Suppose on the contrary that $\nc(F)\leq 2$.

  First suppose that $F$ has no repeated vertex. Thus $F$ is a cycle.
  If $\nc(F) = 0$, then $F$ is a cuff-cycle, every vertex of which is
  not incident to a non-cuff edge, contradicting the fact that $G$ is
  connected with at least two non-cuff edges.  If $\nc(F)=1$ then $F$
  is a contractible cycle that contains exactly one non-cuff edge
  $e$. Thus $e$ is contractible, and Case~1 applies.  If $\nc(F)=2$
  then $F$ is a contractible cycle containing exactly two non-cuff
  edges $e$ and $f$. Thus $e$ and $f$ are homotopic. Hence there is a
  non-trivial equivalence class, and Case~2 applies.

  Now assume that $F$ contains a repeated vertex $v$. Let
$$F=(x_1,x_2,\dots,x_{i-1},x_i=v,x_{i+1},x_{i+2},\dots,x_{j-1},x_j=v)\enspace.$$ All of $x_1,x_{i-1},x_{i+1},x_{j-1}$ are
adjacent to $v$.  Since $x_1\neq x_{j-1}$ and $x_{i-1}\neq x_{i+1}$
and $\deg(v)\leq 3$, we have $x_{i+1}=x_{j-1}$ or
$x_1=x_{i-1}$. Without loss of generality, $x_{i+1}=x_{j-1}$. Thus the
path $x_{i-1}vx_1$ is in the boundary of the cuff-cycle $C$ that
contains $v$.  Moreover, the edge $vx_{i+1}=vx_{j-1}$ counts twice in
$\nc(F)$. Since $\nc(F)\leq 2$, every edge on $F$ except $vx_{i+1}$
and $vx_{j-1}$ is a cuff-edge.  Thus every edge in the walk $v, x_1,
x_2,\dots,x_{i-1}, x_i=v$ is in $C$, and hence
$v,x_1,x_2,\dots,x_{i-1},x_i = v$ is the cycle $C$.  Similarly,
$x_{i+1},x_{i+2},\dots,x_{j-2},x_{j-1}=x_{i+1}$ is a cycle $C'$
bounding some other cuff.  Hence $vx_{i+1}$ is the only non-cuff edge
incident to $C$, and the same for $C'$. Therefore $G$ consists of two
cuff-cycles joined by a non-cuff edge, and Case~3 applies.



Therefore, $\nc(F) \geq 3$, as claimed.

Let $n := |V(G)|$, $m := |E(G)|$, and $f$ be the number of faces of
$G$.  It follows from Euler's formula that
\begin{equation}
  \label{eq:Euler}
  n - m + f + c  = 2 - g.
\end{equation}
Every non-cuff edge appears exactly twice in faces of $G$ (either
twice in the same face, or once in two distinct faces). Thus
\begin{equation}
  \label{eq:bound_t}
  2t = \sum_{F \text{ face of } G} \nc(F) \geq 3f.
\end{equation}
Since $n = m - t$, we deduce from~\eqref{eq:Euler}
and~\eqref{eq:bound_t} that
$$
t = f + c + g - 2 \leq \tfrac{2}{3}t + c + g - 2\enspace.
$$
Thus $t \leq 3(c + g)$, and $m \leq 9t \leq 27(c + g)$ by
\eqref{eq:total_vs_non-cuff}. This allows us, in turn, to bound the
number of edges in $G[4k]$:
$$
|E(G[4k])| = \tbinom{4k}{2}n + (4k)^2m \leq (4k)^2 \cdot 2m \leq 54
(4k)^2 (c + g) \leq 2 (24k)^2 (c + g).
$$
Since $H$ is a minor of $G[4k]$, we have $|E(H)| \leq |E(G[4k])|$.
Thus the minimum degree $d$ of $H$ can be upper bounded as follows:
$$
2|E(H)| \geq d|V(H)| \geq d^{2},
$$
and hence
$$
d \leq \sqrt{2|E(H)|} \leq \sqrt{2|E(G[4k])|} \leq \sqrt{2\cdot 2
  (24k)^2 (c + g)} = \f{k}{c}{g},
$$
as desired.
\end{proof}

Now we put everything together and prove
Theorem~\ref{thm:upper_bound}.

\begin{proof}[Proof of Theorem~\ref{thm:upper_bound}.]
  Let $G \in \G(g,p,k)$.  By Lemma~\ref{lem:combined_lemma}, there
  exists a $(g,p)$-embedded graph $G'$ with $$\eta(G) \leq
  \eta(\lex{G'}{k+1})+\sqrt{6g} + 4\enspace.$$ Let
  $t:=\eta(\lex{G'}{k+1})$. Thus $K_{t}$ is a {\wminor{(k+1)}} of $G'$
  by Lemma~\ref{lem:lex_prod_k_minor}. Lemma~\ref{lem:main_tool} with
  $H=K_{t}$ implies that
$$
\eta(\lex{G'}{k+1}) - 1 = t - 1 \leq 48(k+1)\sqrt{g+p}\enspace.
$$ 
Hence $\eta(G) \leq 48(k+1)\sqrt{g+p} +\sqrt{6g} + 5$, as desired.
\end{proof}

\section{Constructions}
\label{sec:Constructions}

This section describes constructions of graphs in $\G(g,p,k,a)$ that
contain large complete graph minors.  The following lemma, which in
some sense, is converse to Lemma~\ref{lem:combined_lemma} will be
useful.


\begin{lemma}
  \label{lem:ConstructVortex}
  Let $G$ be a graph embedded in a surface with Euler genus at most
  $g$. Let $F_1 ,\dots,F_p$ be pairwise vertex-disjoint facial cycles of $G$,
  such that $V (F_1 ) \cup\dots\cup V (F_p )=V(G)$. Then for all
  $k\geq1$, some graph in $\G(g,p,k)$ contains $\lex{G}{k}$ as a
  minor.
\end{lemma}

\begin{proof}
  Let $G'$ be the embedded multigraph obtained from $G$ by replacing
  each edge $vw$ of $G$ by $k^2$ edges between $v$ and $w$ bijectively
  labeled from $\{(i,j):i,j\in[1,k]\}$.  Embed these new edges
  `parallel' to the original edge $vw$.  Let $H_0$ be the splitting of
  $G'$ at $F_1,\dots,F_p$.  Edges in $H_0$ inherit their label in
  $G'$.  For each $\ell\in[1,p]$, let $J_\ell$ be the face of $H_0$
  that corresponds to $F_{\ell}$.

  \begin{figure}[!htb]
    \begin{center}
      \includegraphics[width=\textwidth]{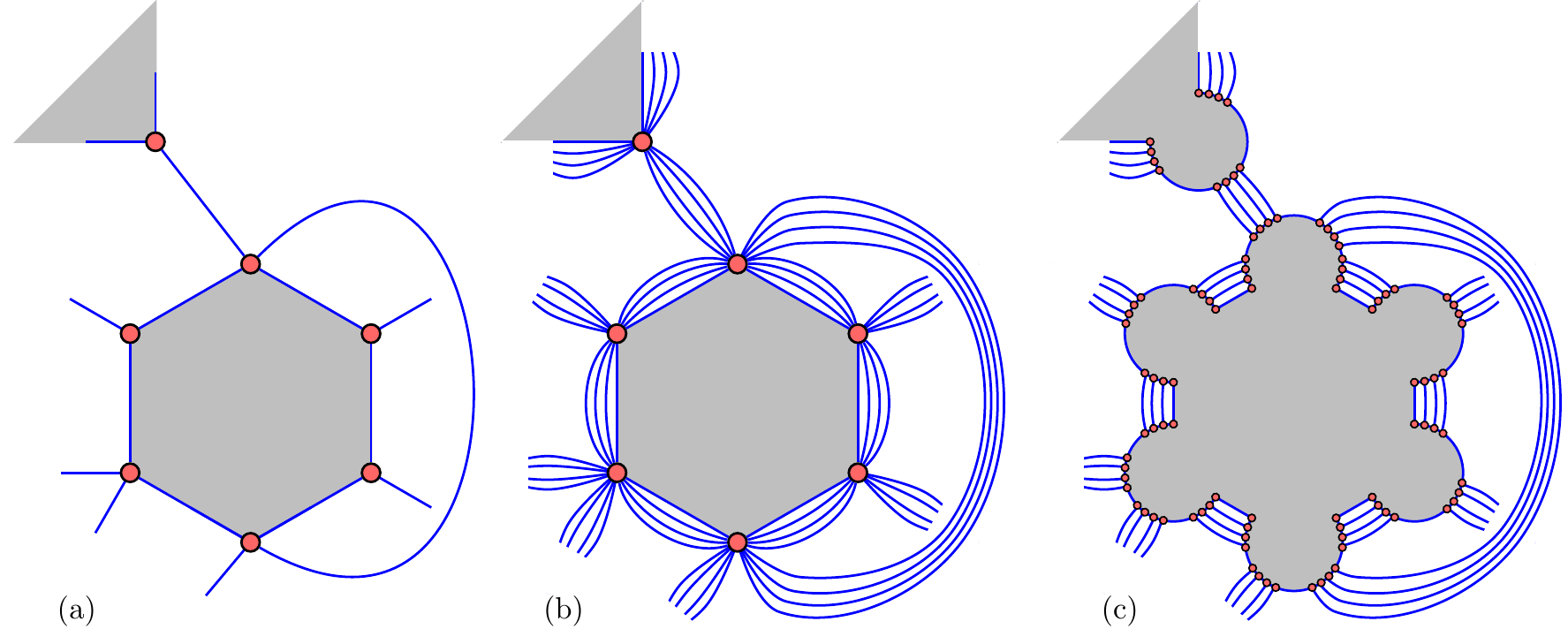}
    \end{center}
    \caption{Illustration for Lemma~\ref{lem:ConstructVortex}: (a)
      original graph $G$, (b) multigraph $G'$, (c) splitting $H_0$ of
      $G'$.}
    \label{fig:ConstructVortex}
  \end{figure}

  Let $H_{\ell}$ be the graph with vertex set
  $V(J_{\ell})\cup\{(v,i):v\in V(F_{\ell}),i\in[1,k]\}$, where:
  \begin{enumerate}[(a)]
  \item each vertex $x$ in $J_{\ell}$ that belongs to a vertex $v$ in
    $F_{\ell}$ is adjacent to each vertex $(v,i)$ in $H_{\ell}$, and
  \item vertices $(v,i)$ and $(w,j)$ in $H_{\ell}$ are adjacent if and
    only if $v=w$ and $i\neq j$.
  \end{enumerate}

  We now construct a circular decomposition $\{B\ang{x}:x\in
  V(J_{\ell})\}$ of $H_{\ell}$ with perimeter $J_{\ell}$.  For each
  vertex $x$ in $J_{\ell}$ that belongs to a vertex $v$ in $F_{\ell}$,
  let $B\ang{x}$ be the set $\{x\} \cup \{(v,i):i\in[1,k]\}$ of
  vertices in $H_{\ell}$.  Thus $|B\ang{x}|\leq k+1$.  For each
  type-(a) edge between $x$ and $(v,i)$, the endpoints are both in bag
  $B\ang{x}$.  For each type-(b) edge between $(v,i)$ and $(v,j)$ in
  $H_{\ell}$, the endpoints are in every bag $B\ang{x}$ where $x$
  belongs to $v$.  Thus the endpoints of every edge in $H_{\ell}$ are
  in some bag $B\ang{x}$.  Thus $\{B\ang{x}:x\in V(J_{\ell})\}$ is a
  circular decomposition of $H_{\ell}$ with perimeter $J_{\ell}$ and
  width at most $k$.

  Let $H$ be the graph $H_0\cup H_1\cup\dots\cup H_p$. Thus
  $V(H_0)\cap V(H_{\ell})=V(J_{\ell})$ for each $\ell\in[1,p]$.  Since
  $J_1,\dots,J_p$ are pairwise vertex-disjoint facial cycles of $H_0$,
  the subgraphs $H_1,\dots,H_p$ are pairwise vertex-disjoint.  Hence
  $H$ is $(g,p,k)$-almost embeddable.

  To complete the proof, we now construct a model
  $\{D_{v,i}:\LexVertex{v}{i}\in V(\lex{G}{k})\}$ of $\lex{G}{k}$ in
  $H$, where $\LexVertex{v}{i}$ is the $i$-th vertex in the $k$-clique
  of $\lex{G}{k}$ corresponding to $v$. Fix an arbitrary total order
  $\preceq$ on $V(G)$.  Consider a vertex $\LexVertex{v}{i}$ of
  $\lex{G}{k}$.
  Say $v$ is in face $F_{\ell}$.  Add the vertex $(v,i)$ of $H_{\ell}$
  to $D_{v,i}$.  For each edge $\LexVertex{v}{i}\LexVertex{w}{j}$ of
  $\lex{G}{k}$ with $v\prec w$, by construction, there is an edge $xy$
  of $H_0$ labeled $(i,j)$, such that $x$ belongs to $v$ and $y$
  belongs to $w$.  Add the vertex $x$ to $D_{v,i}$.  Thus $D_{v,i}$
  induces a connected star subgraph of $H$ consisting of type-(a)
  edges in $H_{\ell}$.  Since every vertex in $J_{\ell}$ is incident
  to at most one labeled edge, $D_{v,i}\cap D_{w,j}=\emptyset$ for
  distinct vertices $\LexVertex{v}{i}$ and $\LexVertex{w}{j}$ of
  $\lex{G}{k}$.

  Consider an edge $\LexVertex{v}{i}\LexVertex{w}{j}$ of $\lex{G}{k}$.
  If $v=w$ then $i\neq j$ and $v$ is in some face $F_{\ell}$, in which
  case a type-(b) edge in $H_{\ell}$ joins the vertex $(v,i)$ in
  $D_{v,i}$ with the vertex $(w,j)$ in $D_{w,j}$.  Otherwise, without
  loss of generality, $v\prec w$ and by construction, there is an edge
  $xy$ of $H_0$ labeled $(i,j)$, such that $x$ belongs to $v$ and $y$
  belongs to $w$.  By construction, $x$ is in $D_{v,i}$ and $y$ is in
  $D_{w,j}$.  In both cases there is an edge of $H$ between $D_{v,i}$
  and $D_{w,j}$.  Hence the $D_{v,i}$ are the branch sets of a
  $\lex{G}{k}$-model in $H$.
\end{proof}



Our first construction employs just one vortex and is based on an
embedding of a complete graph.

\begin{lemma}
  \label{lem:OneVortexConstruction}
  For all integers $g\geq0$ and $k\geq1$, there is an integer $n\geq
  k\sqrt{6g}$ such that $K_n$ is a minor of some $(g,1,k)$-almost
  embeddable graph.
\end{lemma}

\begin{proof}
  The claim is vacuous if $g=0$. Assume that $g\geq1$.  The map color
  theorem \citep{Ringel74} implies that $K_m$ triangulates some
  surface if and only if $m\bmod{6}\in\{0,1,3,4\}$, in which case the
  surface has Euler genus $\tfrac{1}{6}(m-3)(m-4)$. It follows that
  for every real number $m_0\geq 2$ there is an integer $m$ such that
  $m_0\leq m\leq m_0+2$ and $K_m$ triangulates some surface of Euler
  genus $\tfrac{1}{6}(m-3)(m-4)$.  Apply this result with
  $m_0=\sqrt{6g}+1$ for the given value of $g$. We obtain an integer
  $m$ such that $\sqrt{6g}+1\leq m\leq\sqrt{6g}+3$ and $K_m$
  triangulates a surface $\Sigma$ of Euler genus
  $g':=\tfrac{1}{6}(m-3)(m-4)$. Since $m-4<m-3\leq\sqrt{6g}$, we have
  $g'\leq g$.  Every triangulation has facewidth at least 3. Thus,
  deleting one vertex from the embedding of $K_{m}$ in $\Sigma$ gives
  an embedding of $K_{m-1}$ in $\Sigma$, such that some facial cycle
  contains every vertex.  Let $n:=(m-1)k\geq k\sqrt{6g}$.
  Lemma~\ref{lem:ConstructVortex} implies that $\lex{K_{m-1}}{k}\cong
  K_n$ is a minor of some $(g',1, k)$-almost embeddable graph.
\end{proof}


Now we give a construction based on grids.  Let $L_n$ be the $n\times
n$ planar grid graph.  This graph has vertex set $[1,n]\times[1,n]$
and edge set $\{(x,y)(x',y'):|x-x'|+|y-y'|=1\}$.  The following lemma
is well known; see \citep{Wood-ProductMinor}.

\begin{lemma}
  \label{lem:MinorInGridProduct}
  $K_{nk}$ is a minor of $\lex{L_n}{2k}$ for all $k\geq1$.
\end{lemma}

\begin{proof}
  For $x,y\in[1,n]$ and $z\in[1,2k]$, let $(x,y,z)$ be the $z$-th
  vertex in the $2k$-clique corresponding to the vertex $(x,y)$ in
  $\lex{L_n}{2k}$.  For $x\in[1,n]$ and $z\in[1,k]$, let $B_{x,z}$ be
  the subgraph of $\lex{L_n}{2k}$ induced by
  $\{(x,y,2z-1),(y,x,2z):y\in[1,n]\}$.  Clearly $B_{x,z}$ is
  connected. For all $x,x'\in[1,n]$ and $z,z'\in[1,k]$ with
  $(x,z)\neq(x',z')$, the subgraphs $B_{x,z}$ and $B_{x',z'}$ are
  disjoint, and the vertex $(x,x',2z-1)$ in $B_{x,z}$ is adjacent to
  the vertex $(x,x',2z')$ in $B_{x',z'}$.  Thus the $B_{x,z}$ are the
  branch sets of a $K_{nk}$-minor in $\lex{L_n}{2k}$.
\end{proof}




\begin{lemma}
  \label{lem:ManyVortexConstruction}
  For all integers $k\geq2$ and $p\geq1$,
  there is an integer $n\geq\frac{2}{3\sqrt{3}}k\sqrt{p}$, such that
  $K_n$ is a minor of some $(0,p,k)$-almost embeddable graph.
\end{lemma}

\begin{proof}
  Let $m:=\floor{\sqrt{p}}$ and $\ell:=\floor{\frac{k}{2}}$.
  Let
  $n:=2m\ell\geq2\cdot\sqrt{\frac{p}{3}}\cdot\frac{k}{3}=\frac{2}{3\sqrt{3}}k\sqrt{p}$.
  For $x,y\in[1,m]$, let $F_{x,y}$ be the face of $L_{2m}$ with vertex
  set $\{(2x-1,2y-1),(2x,2y-1),(2x,2y),(2x-1,2y)\}$.  There are $m^2$
  such faces, and every vertex of $L_{2m}$ is in exactly one such
  face.  By Lemma~\ref{lem:MinorInGridProduct}, $K_n$ is a minor of
  $\lex{L_{2m}}{2\ell}$.  Since $L_{2m}$ is planar, by
  Lemma~\ref{lem:ConstructVortex}, $K_n$ is a minor of some
  $(0,m^2,2\ell)$-almost embeddable graph.  The result follows since
  $p\geq m^2$ and $k\geq 2\ell$.
\end{proof}





The following theorem summarizes our constructions of almost
embeddable graphs containing large complete graph minors.

\begin{theorem}
  \label{thm:CombinedConstructions}
  For all integers $g\geq 0$ and $p\geq1$ and $k\geq2$, there is an
  integer $n\geq \frac{1}{4} k\sqrt{p+g}$, such that $K_n$ is a minor
  of some $(g,p,k)$-almost embeddable graph.
\end{theorem}

\begin{proof}
  First suppose that $g\geq p$. By
  Lemma~\ref{lem:OneVortexConstruction}, there is an integer $n\geq
  k\sqrt{6g}$, such that $K_n$ is a minor of some $(g,1,k)$-almost
  embeddable graph, which is also $(g,p,k)$-embeddable (since
  $p\geq1$). Since $n\geq k\sqrt{3p+3g}> \frac{1}{4}k\sqrt{p+g}$, we
  are done.
 
  Now assume that $p>g$. By Lemma~\ref{lem:ManyVortexConstruction},
  there is an integer $n\geq\frac{2}{3\sqrt{3}}k\sqrt{p}$, such that
  $K_n$ is a minor of some $(0,p,k)$-almost embeddable graph, which is
  also $(g,p,k)$-embeddable (since $g\geq0$). Since $n
  \geq\frac{2}{3\sqrt{3}}k\sqrt{\frac{g}{2}+\frac{p}{2}}
  =\frac{\sqrt{2}}{3\sqrt{3}}k\sqrt{g+p} >\frac{1}{4}k\sqrt{g+p}$, we
  are done.
\end{proof}

Adding $a$ dominant vertices to a graph increases its Hadwiger number
by $a$. Thus Theorem~\ref{thm:CombinedConstructions} implies:

\begin{theorem}
  \label{thm:ApexConstruction}
  For all integers $g,a\geq 0$ and $p\geq1$ and $k\geq2$, there is an
  integer $n\geq a+\frac{1}{4}k\sqrt{p+g}$, such that $K_n$ is a minor
  of some graph in $\G(g,p,k,a)$.
\end{theorem}

Corollary~\ref{cor:CliqueMinor} and Theorem~\ref{thm:ApexConstruction}
together prove Theorem~\ref{thm:Main}.

\subsubsection*{Acknowledgments} Thanks to Bojan Mohar for instructive
conversations.


\end{document}